\documentclass[12pt]{amsart}
\usepackage{amsfonts}

\newtheorem{theorem}{Theorem}[section]
\newtheorem{corollary}[theorem]{Corollary}

\newtheorem{lemma}[theorem]{Lemma}
\newtheorem{proposition}[theorem]{Proposition}
\newtheorem{example}[theorem]{Example}

\newtheorem{remark}[theorem]{Remark}

\def\Cc{\hbox{\sf C\kern -.47em {\raise .48ex \hbox{$\scriptscriptstyle |$}}
   \kern-.5em {\raise .48ex \hbox{$\scriptscriptstyle |$}} }}

\newcommand{\si}{\sigma}

\newcommand{\be}{\begin{equation}}
\newcommand{\ee}{\end{equation}}

\newcommand{\NN}{\mathbb{N}}

\newcommand{\e}{\varepsilon}
\newcommand{\al}{\alpha}

\def\si{\sigma}

\begin{document}

\baselineskip 5.8mm

\title[On the Bonsall cone spectral radius and the approximate point spectrum ]
{On the Bonsall cone spectral radius and the approximate point spectrum}

\author{Vladimir M\"uller, Aljo\v{s}a Peperko}
\date{\thanks{} \today}

\begin{abstract}
\baselineskip 7mm

We study the Bonsall cone spectral radius and the approximate point spectrum
of  (in general non-linear) positively homogeneous, bounded and supremum preserving maps, defined on a max-cone in a given 
normed vector lattice. We prove that the Bonsall cone spectral radius of such maps  is always included in its approximate point spectrum. Moreover, the approximate point spectrum always contains a (possibly trivial) interval. Our results apply to a large class of (nonlinear) max-type operators.

We also generalize a known result that the spectral radius of a positive (linear) operator on a Banach lattice is contained in the approximate point spectrum. Under additional generalized compactness type assumptions our results imply Krein-Rutman type results. 

\end{abstract}

\maketitle

\noindent
{\it Math. Subj.  Classification (2010)}: 47H07, 47J10, 47H10, 47H08, 47B65, 47A10. 
\\

\section{Introduction}
Max-type operators (and corresponding max-plus type operators and their tropical versions known also as Bellman operators) arise in a large field of problems from the theory of differential and difference equations,  mathematical physics, optimal control problems, discrete mathematics, turnpike theory, mathematical economics, mathematical biology,  games and controlled Markov processes, generalized solutions of the Hamilton-Jacobi-Bellman differential equations, continuously observed and controlled quantum systems, discrete and continuous dynamical systems, ... (see e.g. \cite{MN02}, \cite{KM97},  \cite{LM05}, \cite{LMS01},  \cite{AGN} and the references cited there). The eigenproblem of such operators obtained so far substantial attention due to its applicability in the above mentioned problems (see e.g. \cite{MN02}, \cite{KM97}, \cite{AGN}, \cite{AGW04} \cite{LN12}, \cite{AG13}, \cite{MN10}, \cite{BCOQ92}, \cite{B98}, \cite{MP15}, \cite{BGC-G09}, \cite{MP12}, \cite{S07} and the references cited there). However, there seems to be a lack of more general treatment of spectral theory for such operators, eventhough the spectral theory for nonlinear operators on Banach spaces is already quite well developed
(see e.g. \cite{APV04}, \cite{APV00}, \cite{AGV02}, \cite{F97}, \cite{FV75}, \cite{FMV78}, \cite{GM77},
\cite{SV00} and the references cited there). One of the reasons for this might lie in the fact that these operators behave nicely on a suitable subcone (or subsemimodule), but less nicely on the whole (Banach) space. Therefore it appears, that it is not trivial to directly apply this known non-linear spectral theory to obtain satisfactory information on a restriction to a given cone of a max-type operator.   
 The Bonsall cone spectral radius  plays the role of the spectral radius in this theory (see e.g.  \cite{MN02}, \cite{MN10}, \cite{AGN}, \cite{LN11}, \cite{Gr15}, \cite{MP15} and the references cited there). 

In this article we study the Bonsall cone spectral radius  and the approximate point spectrum
of  positively homogeneous, bounded and supremum preserving maps, defined on a max-cone in a given 
normed vector lattice. We prove that the Bonsall cone spectral radius of such maps  is always included in its approximate point spectrum. Moreover, the approximate point spectrum always contains a (possibly trivial) interval. Our results apply to a large class of max-type operators (and their isomorphic versions). Our main interests are  results on suitable cones in Banach spaces and Banach lattices. However, since the completeness of the norm does not simplify our proofs, 
we state 
our results in the setting of normed spaces and normed 
vector lattices. Under suitable generalized compactness type assumptions our results imply Krein-Rutman type results. 

The paper is organized as follows. In Section 2 we recall basic definitions and facts that we will need in our proofs. In Section 3 we prove our results in the setting of max-cones in normed vector lattices, while in Section 4 we apply our techniques in the setting of normal convex cones in normed spaces. The main results of Section 3 are Theorem \ref{main_lattice}  and its generalization Theorem \ref{appl_to_matrices} and the main result of Section 4 are Theorems \ref{main_normal} and  \ref{appl_to_matrices_lin}.


\section{Preliminaries}

A subset $C$ of a real vector space $X$ is called a cone (with vertex 0) if 
$tC \subset C$ for all $t \ge 0$, where $tC =\{tx : x \in C \} $. A map $T: C \to C$ is called positively homogeneous
(of degree 1) if $T(tx) = tT(x)$ for all $t \ge 0$ and $x \in C$. We say that the cone
$C$ is pointed if $C \cap (- C) = \{0\}$. 

A convex pointed cone $C$ of $X$ induces on $X$ a
partial ordering $\le$, which is defined by $x  \le y$ if and only if  $y - x \in C$. In this case $C$ is denoted by $X_+$ and $X$ is called an ordered vector space. If, in addition, $X$ is a normed space then it is called an ordered normed space.  If, in addition, the norm is complete, then  $X$ is called an ordered Banach space.

A convex cone $C$ of $X$ is  called a wedge. A wedge induces on $X$  (by the above relation) a vector preordering $\le$ (which is reflexive, transitive, but not necessary antisymmetric).

We say that the cone $C$ is proper if it is closed, convex and pointed. A cone $C$ of a normed space $X$  is called normal if there exists a constant $M$ such
that $\|x\| \le M \|y\|$ whenever $ x \le y$, $x,y\in C$. A convex and pointed  cone $C=X_+$ of an ordered normed space $X$ is normal if and only if there exists an equivalent monotone norm $||| \cdot|||$ on $X$, i.e.,  
$|||x||| \le |||y|||$ whenever $0 \le x \le y$ (see e.g. \cite[Theorem 2.38]{AT07}).
Every proper cone $C$ in a finite dimensional
Banach space  is necessarily normal.

If $X$ is a normed linear space, then a cone $C$ in $X$ is said to be complete if it is a complete metric space in the topology induced by $X$. In the case when $X$ is a Banach space this is equivalent to $C$ being closed in $X$.

If $X$ is an ordered vector space, then a cone $C\subset X_+$ is called a max-cone if for every pair  $x,y\in C$ there exists a supremum $x\vee y$  (least upper bound) in $C$. 
We consider here on $C$ an order inherited from $X_+$.
A map $T:C\to C$ preserves finite suprema on $C$ if $T(x\vee y)=Tx\vee Ty\quad(x,y\in C$). If $T:C\to C$ preserves finite suprema, then it is monotone (order preserving) on $C$, i.e., $Tx \le Ty$ whenever $x\le y$, $x,y \in C$  .

An ordered vector space  $X$ is called a vector lattice (or a Riesz space) if every two vectors $x,y \in X$ have a supremum and  infimum (greatest lower bound) in $X$. A positive cone $X_+$ of a vector lattice $X$ is called a lattice cone.

Note that by \cite[Corollary 1.18]{AT07} a pointed convex cone $C=X_+$ of an ordered vector space $X$ is a lattice cone for the vector subspace $C-C$ generated by $C$ in $X$, if and only if  $C$ is a max cone (in this case a supremum of $x, y\in C$ exists in $C$ if only if it exists in $X$; and suprema coincide). Moreover, if $x,y,z,u \in C$, then
$$(x-y)\vee (z-u)= (x+u) \vee (y+z) - (y+u)$$  
holds in $C-C$.

If $X$ is a vector lattice, then the absolute value of $ x\in X$ is defined by $|x|= x \vee (-x)$. A vector lattice is called a normed vector lattice (a normed Riesz space) if $|x| \le |y|$ implies $\|x\| \le \|y\|$. A complete normed vector lattice is called a Banach lattice. A positive cone $X_+$ of a normed vector lattice $X$ is  proper and normal. 

In a vector lattice $X$  the following Birkhoff's inequality for $x_1,\dots,x_n,y_1,\dots,y_n\in X$ holds:
\be
|\bigvee_{j=1}^n x_j-\bigvee_{j=1}^n y_j| \le \sum_{j=1}^n |x_j-y_j|.
\label{Birk_inq}
\ee
For the theory of  vector lattices, Banach lattices, cones, wedges, operators on cones and applications e.g. in financial mathematics we refer the reader 
to 
 \cite{AA02}, \cite{AT07}, 
 \cite{AB85}, \cite{W99}, \cite{ABB90}, \cite{LT96}, \cite{JM14}, \cite{AB06}  
 and the references cited there. 

Let $X$ be a normed space and $C \subset X$ a non-zero cone. Let $T:C \to C$ be positively homogeneous and bounded, i.e., 
$$\|T\|:=\sup \left \{\frac{\|Tx\|}{\|x\|} :x\in C, x\neq 0\right \}<\infty.$$
It is easy to see that
$\|T\|=\sup\{\|Tx\|:x\in C, \|x\|\le 1 \}$ and $\|T^{m+n}\|\le \|T^m\|\cdot \|T^n\|$ for all $m,n\in\NN$. It is well known that this implies that the limit $\lim_{n\to\infty}\|T^n\|^{1/n}$ exists and is equal to $\inf_n \|T^n\|^{1/n}$. The limit $r(T):=\lim_{n\to\infty}\|T^n\|^{1/n}$ is called the Bonsall cone spectral radius of $T$. The approximate point spectrum $\sigma_{ap}(T)$ of $T$ is defined as the set of all $s\ge 0$ such that $\inf\{\|Tx-sx\|:x\in C,\|x\|=1\}=0$.

For $x\in C$ define the local cone spectral radius by $r_x(T):=\limsup_{n\to\infty}\|T^nx\|^{1/n}$. Clearly $r_x(T)\le r(T)$ for all $x\in C$. It is known that the equality 
\be 
\sup\{r_x(T):x\in C\}=r(T)
\label{eq}
\ee
is not valid in general. In \cite{MN02} there is an example of a proper cone $C$ in a Banach space $X$ and  a positively homogeneous and continuous (hence bounded)  map $T : C \to C$ such that $\sup\{r_x(T):x\in C\}< r(T)$. A recent example of such kind, where  $T$ is in addition monotone, is obtained in \cite[Example 3.1]{Gr15}. However, if $C$ is a normal, complete, convex and pointed cone in a normed space $X$  and  $T : C \to C$ is  positively homogeneous, monotone  and continuous, then 
\cite[Theorem 3.3]{MN10}, 
 \cite[Theorem 2.2]{MN02} and \cite[Theorem 2.1]{Gr15} ensure that (\ref{eq}) is valid.


If $X$ is a Banach lattice,  $C\subset X_+$  a max-cone and $T:C\to C$ a mapping which is bounded, positively homogeneous and preserves finite suprema, then the equality (\ref{eq}) is not necessary valid as the following example shows.
\begin{example}
{\rm Let $X=l^2$ 
with a standard orthonormal basis $\{e_1,e_2,\dots \}$. Let $C=\{\bigvee_{j=1}^n\al_je_j:n\ge 1, \al_1,\dots\al_n\ge 1\}$.
Define $T:C\to C$ by $T(\bigvee_{j=1}^n\al_je_j)=\bigvee _{j=1}^{n-1}\al_{j+1}e_j$ (the backward shift). It is easy to see that $r(T)=1$ and $r_x(T)=0$ for each $x\in C$. It also holds that $\sigma_{ap}(T)=[0,1]$.
}
\label{example}
\end{example}
Some additional examples of maps for which (\ref{eq}) is not valid can be found in \cite{Gr15}.

Let $C$ be a cone in a normed space $X$ and $T:C \to C$. Then $T$ is called  Lipschitz if there exists $L >0$ such that $\|Tx-Ty \|\le L\|x-y\| $ for all $x,y \in C$. 

\section{Results on max-cones in normed vector lattices}

As noted in Example \ref{example} it may happen that $\sup\{r_x(T):x\in C\}<r(T)$ in the case when $C$ is a max-cone. We will prove in Theorem \ref{main_lattice} that the Bonsall cone spectral radius of  a  bounded, positively homogeneous, finite suprema preserving  mapping $T:C\to C$, defined on a max-cone $C$ in a normed vector lattice, is contained in its approximate point spectrum. Moreover, we will show that the interval 
 $[\sup\{r_x(T):x\in C\} , r(T)]$ is included in  $\sigma_{ap}(T)$ for such maps $T$.  






We shall need the following three lemmas.

\begin{lemma}
\label{Birk}
Let $X$ be a normed vector lattice and let $x_1,\dots,x_n,y_1,\dots,y_n\in X$. Then
$$
\Bigl\|\bigvee_{j=1}^n x_j-\bigvee_{j=1}^n y_j\Bigr\|\le
\sum_{j=1}^n \|x_j-y_j\|.
$$
\end{lemma}

\begin{proof}
By (\ref{Birk_inq}) we have
$$
\Bigl\|\bigvee_{j=1}^n x_j-\bigvee_{j=1}^n y_j\Bigr\|=
\Bigl\|\, |\bigvee_{j=1}^n x_j-\bigvee_{j=1}^n y_j|\,\Bigr\|
\le
\Bigl\|\sum_{j=1}^n |x_j-y_j|\, \Bigr\|
$$
$$
\le 
\sum_{j=1}^n \|\,|x_j-y_j|\,\|=
\sum_{j=1}^n \|x_j-y_j\|,
$$
which completes the proof.
\end{proof}

\begin{lemma} Let $X$ be a vector lattice and $x_j, y_j \in X$ 
for $j= 1, \ldots, n$. Then 
\be
\bigvee_{j=1}^n x_j - \bigvee_{j=1}^n y_j  \le \bigvee_{j=1}^n (x_j -y_j) .
\label{supineq}
\ee
If, in addition, $X$ is a normed vector lattice and  $x_j \ge y_j \ge 0$ 
for $j= 1, \ldots, n$, then 
\be
\|\bigvee_{j=1}^n x_j - \bigvee_{j=1}^n y_j \| \le \|\bigvee_{j=1}^n (x_j -y_j) \|.
\label{supineq2}
\ee
\label{sup}
\end{lemma}
\begin{proof} We have
$$\bigvee_{j=1}^n x_j    = \bigvee_{j=1}^n (x_j -y_j +y_j) \le  \bigvee_{j=1}^n (x_j -y_j) + \bigvee_{j=1}^n y_j.$$
which proves (\ref{supineq}).
If  $X$ is a normed vector lattice and  $x_j \ge y_j \ge 0$ for $j= 1, \ldots, n$, then $\bigvee_{j=1}^n x_j  \ge \bigvee_{j=1}^n y_j $ and this implies (\ref{supineq2}).
\end{proof}

\begin{lemma} Let $X$ be a normed space and let $C\subset X$ be a non-zero cone.
 If $T:C \to C$ is positively homogeneous and Lipschitz, then $r(T) \ge t$ for all $t \in \sigma _{ap} (T)$. 
\label{spdominate}
\end{lemma}
\begin{proof}  Since $T(0)=0$ and $T$ is Lipschitz it follows that $T$ is also bounded and so $r(T)$ is well defined.
If $t \in \si _{ap}(T)$, then there exists   
a sequence $(x_k)$ of unit vectors such that 
$\lim_{k\to\infty}\|Tx_k-tx_k\|=0.$ By induction it follows that
$\lim_{k\to\infty}\|T^j x_k-t^jx_k\|=0$ for all $j\in \NN$.
Indeed, 
$$\|T^j  x_k-t^jx_k\|
\le \|T^j  x_k -T^{j-1}( tx_k)\| + \|T^{j-1}( tx_k) -t^jx_k\| $$
$$\le L^{j-1} \|T x_k - tx_k \|+ t \|T^{j-1} x_k -t^{j-1}x_k\| \to 0$$
as $k\to \infty$, by the induction assumption. Here $L$ denotes  the Lipschitzity constant of $T$.

It follows that
$ \|T^j \| \ge \lim _{k\to \infty}\|T^j x_k \|=t^j$ and so $r(T) \ge t$.
\end{proof}

The following example shows that in Lemma \ref{spdominate} 
we can not replace the property that "$T$ is Lipschitz" by a weaker property that "$T$ is bounded". 
\begin{example} {\rm 
Let $X=\ell^\infty$ with the standard basis $x_n, y_n, z_n \quad(n=1,2,\dots)$.
More precisely, the elements of $X$ are formal sums
$$
x=\sum_{n=1}^\infty (\alpha_n x_n +\beta_n y_n + \gamma_n z_n)
$$
with real coefficients $\alpha_n, \beta_n, \gamma_n$ such that
$$
\|x\|:= \sup\{|\alpha_n|, |\beta_n|, |\gamma_n|: n=1,2,\dots\}<\infty.
$$
Then $X$ is a Banach lattice with the natural order.

Let 
$$
C=\Bigl\{
\sum_{j=1}^\infty (\alpha_j x_j+\beta_j y_j+\gamma_j z_j)\in X: \alpha_j, \gamma_j\ge 0, \beta_j=j\alpha_j+j\gamma_j\ \hbox{ for all  } j\Bigr\}.
$$
Then $C$ is a closed max-cone (moreover convex and normal).
Let $T:C\to C$ be defined by
$$
T\Bigl(\sum_{j=1}^\infty (\alpha_j x_j+\beta_j y_j+\gamma_j z_j)\Bigr)=
\sum_{j=1}^\infty(j\alpha_j y_j+\alpha_j z_j).
$$
Clearly $\|T\|\le 1$, $T$ is positively homogeneous and preserves (all) suprema.

For $k\in \mathbb{N}$ let $u_k=k^{-1}x_k+y_k$. Then $\|u_k\|=1$ and $Tu_k=y_k+k^{-1}z_k$. So $\|Tu_k-u_k\|=k^{-1}$ and $1\in\sigma_{ap}(T)$.
On the other hand $T^2=0$ and so $r(T)=0$.
 
Note that  $T$ is not Lipschitz, since $\|Tu_k-u_k\|=k^{-1}$ but $\|T^2u_k-Tu_k\|=\|Tu_k\|=1$.
}
\label{good_stuff}
\end{example}
The following technical lemma is essentially needed in the proofs of our main results Theorem \ref{main_lattice}, Theorem \ref{appl_to_matrices}, Theorem \ref{main_normal} and Theorem \ref{appl_to_matrices_lin}.
\begin{lemma}
\label{numbers}
Let $\e>0$ and $K\ge 1$. Then there exists $n\in\NN$ with the following property:
if $(\al_k)_{k=0}^\infty$ is a sequence of real numbers such that $0\le\al_k\le K^k$ for all $k$, $\al_n\ge 1/2$ and $\limsup_{k\to\infty}\al_k^{1/k}\le 1$, then there exist $m\in\NN$ and nonnegative numbers $\beta_k, (k=0,1,\dots),$ such that
$$
\beta_0\le\e ,
$$
$$
|\beta_{k+1}-\beta_k|\le 2\e ,
$$
$$ \beta _k < \beta _{k+1} \quad ({\it for \ all \ } k=0, 1, \ldots , m-1),$$
$$ \beta _k > \beta _{k+1} \quad ({\it for \ all \ } k=m, m+1, \ldots ),$$
$$
\al_m\beta_m=1, \al_k\beta_k\le 1\quad ({\it for\  all\  }k),
$$
$$
\lim_{k\to\infty}\al_k\beta_{k+1}=0.
$$
\end{lemma}

\begin{proof}
Choose $m_0\in\NN$ such that $(1+\e)^{m_0}>2\e^{-1}$. Choose $n>m_0$ such that $(1+\e)^n>2K^{m_0} (1+\e)^{m_0}$.
Let $(\al_k)$ be a sequence of nonnegative numbers satisfying the assumptions.

Set $\gamma_k=\alpha_n(1+\e)^{|k-n|}$. Then $\gamma_n=\alpha_n \ge 1/2$.
Let $m\in\NN$ satisfy 
$$
\frac{\al_m}{\gamma_m}=\max_k\frac{\al_k}{\gamma_k}
$$
(such an $m$ exists since $\lim_{k\to\infty}\frac{\al_k}{\gamma_k}=
\al_n^{-1}\lim_{k\to\infty}\frac{\alpha _k}{(1+\e)^{k-n}}=0$).
In particular, we have $\frac{\al_m}{\gamma_m}\ge\frac{\al_n}{\gamma_n}=1$ and $\al_m\ge \gamma_m\ge \al_n\ge 1/2$.

We show that $m\ge m_0$. Suppose the contrary that $m < m_0$. Then $n \ge m$, since otherwise $n <m < m_0$ provides a contradiction. We have $\al_m\le K^m$ and \\ 
$\gamma_m=\al_n(1+\e)^{n-m}$. So 
$$
\frac{\al_m}{\gamma_m}\le
\frac{K^m}{\al_n(1+\e)^{n-m}}<
\frac{2K^{m_0}(1+\e)^{m_0}}{(1+\e)^n}<1,
$$
a contradiction. So $m\ge m_0$.

Set $\beta_k=\al_m^{-1}(1+\e)^{-|k-m|}$. Clearly $\al_m\beta_m=1$ and $\beta_m=\al_m^{-1}\le 2$.
We have \\
$\beta_0=\al_m^{-1}(1+\e)^{-m}\le 2(1+\e)^{-m_0}<\e$.
Clearly 
$\lim_{k\to\infty}\al_k\beta_{k+1}=0$.

For $k=0,1,\dots,m-1$ we have $\beta_k<\beta_{k+1}$ and
$$
\beta_{k+1}-\beta_{k}=
\al_m^{-1}\bigl( (1+\e)^{-(m-k-1)}-(1+\e)^{-(m-k)}\bigr)=
\al_m^{-1}(1+\e)^{k-m}(1+\e-1)\le
\e\al_m^{-1}\le 2\e.
$$
Similarly, for $k=m,m+1,\dots$ we have 
$\beta_k >\beta_{k+1}$ and
$$
\beta_{k}-\beta_{k+1}=
\al_m^{-1}\bigl( (1+\e)^{-(k-m)}-(1+\e)^{-(k+1-m)}\bigr)=
\al_m^{-1}(1+\e)^{m-k-1} \e\le
2\e.
$$

Finally, for each $k$ we have $\frac{\al_k}{\gamma_k}\le\frac{\al_m}{\gamma_m}$. So 
$$
\al_k\beta_k\le\frac{\al_m\gamma_k\beta_k}{\gamma_m}=
\frac{\al_n(1+\e)^{|k-n|}}{\al_n(1+\e)^{|m-n|}(1+\e)^{|k-m|}}\le 1$$
and the proof is complete.
\end{proof}

The following theorem is one of the main results of this section.
\begin{theorem} 
\label{main_lattice}
Let $X$ be 
a normed vector lattice, 
let $C\subset X_+$ be a non-zero max-cone. 
Let $T:C\to C$ be a mapping which is bounded, positively homogeneous and preserves finite suprema.
Let 
$\sup\{r_x(T):x\in C\}\le t\le r(T)$. Then $t\in\sigma_{ap}(T)$.

In particular, $ r(T)\in\sigma_{ap}(T)$. 
\end{theorem}
\begin{proof}
If $t=0$, then for each $x\in C$, $\|x\|=1$ we have  $\lim_{n\to\infty}\|T^nx\|^{1/n}=0$. For each $\e>0$ there exists $k\ge 0$ such that $\|T^{k+1}x\|<\e \|T^kx\|$.
If $u=\frac{T^kx}{\|T^kx\|}$ then $\|u\|=1$ and $\|Tu\|<\e$.

So without loss of generality we may assume that $t=1$.

\medskip
We distinguish two cases:

\noindent{\bf I.}
Suppose that $\sup\{\|\bigvee_{j=0}^n T^jx\|:x\in C, \|x\|\le 1, n\in\NN\} =\infty$. 

Let $k\in\NN$. Find $x_k\in C$, $\|x_k\|\le 1$ and $n_k\in\NN$ such that $\|\bigvee_{j=0}^{n_k} T^jx_k\|>k$.
Find $t_k\in(1,1+k^{-1})$ such that
$t_k^{-n_k-1}> 1/2$. Find $r_k>n_k$ such that $\frac{\|T^{r_k+1}x_k\|}{t_k^{r_k+1}}<1$. Set $y_k:=\bigvee_{j=0}^{r_k}\frac{T^jx_k}{t_k^{j+1}}$. Then $\|y_k\|\ge
\|t_k^{-n_k-1}\bigvee_{j=0}^{n_k} T^jx_k \|\ge k/2$.

Set $u_k=\frac{y_k}{\|y_k\|}$. Then $\|u_k\|=1$ and 
 by Lemma \ref{Birk} it follows
$$
\|Tu_k-u_k\|\le
\|Tu_k-t_ku_k\|+(t_k-1)\|u_k\|=
\|y_k\|^{-1}\Bigl\|\bigvee_{j=0}^{r_k}\frac{T^{j+1}x_k}{t_k^{j+1}}-\bigvee_{j=0}^{r_k}\frac{T^{j}x_k}{t_k^{j}}\Bigr\|+ (t_k-1)
$$
$$
\le
\|y_k\|^{-1}\Bigl\|\frac{T^{r_k+1}x_k}{t_k^{r_k+1}}-x_k\Bigr\|+ k^{-1}
\le
2\|y_k\|^{-1}+ k^{-1}\le\frac{5}{k}\to 0
$$
as $k\to\infty$.
Hence $1\in\sigma_{ap}(T)$.

\medskip

\noindent{\bf II.}
Suppose that
$$
M_0:=
\sup\Bigl\{\Bigl\|\bigvee_{j=0}^n T^jx\Bigr\|:x\in C, \|x\|\le 1, n\in\NN\Bigr\}<\infty.
$$

Let $\e>0$ and $K:=\|T\|$. Let $n$ be the number constructed in Lemma \ref{numbers}.  
We have $\|T^n\|\ge r(T^n)=r(T)^n\ge 1$. 
Find $x\in C$ such that $\|x\|=1$ and $\|T^nx\|\ge 1/2$. Let $\alpha_k=\|T^kx\|$.

By Lemma \ref{numbers}, there exist $m\in\NN$ and nonnegative numbers $\beta_k$ such that $\beta_0\le\e$, $|\beta_{k+1}-\beta_k|\le2\e$, $ \beta _k < \beta _{k+1} $  for  all  $k=0, 1, \ldots , m-1$,
$ \beta _k > \beta _{k+1} $ for  all $ k=m, m+1, \ldots $,
$\al_m\beta_m=1$, $\al_k\beta_k\le 1$ for all $k$ and $\lim_{k\to\infty}\al_k\beta_{k+1}=0$.

Fix $r>m$
such that $\beta_{r}<\e\|T^{r-1}x\|^{-1}$.

Set $u=\bigvee_{k=0}^r\beta_k T^kx$. Since $u\ge \beta_m{T^mx}$, we have $\|u\|\ge \beta_m\|T^mx\|=1$.
By Lemma \ref{Birk} and Lemma \ref{sup} we have
$$
\|Tu-u\|
=
\Bigl\|\bigvee_{k=0}^r\beta_k{T^{k+1}x}-
\bigvee_{k=0}^r\beta_k T^kx\Bigr\|
$$
$$
\le
\Bigl\|\beta_0{x}\Bigr\|+
\Bigl\|\bigvee_{k=1}^{m}{T^kx}\beta_{k}
-\bigvee_{k=1}^{m}{T^kx}\beta_{k-1}\Bigr\|+
\Bigl\|\bigvee_{k=m+1}^{r}{T^kx}\beta_{k-1}
-\bigvee_{k=m+1}^{r}{T^kx}\beta_{k}\Bigr\|+
\Bigl\|\beta_r{T^{r+1}x}\Bigr\|
$$
$$
\le
\e+
\Bigl\|\bigvee_{k=1}^{m}{T^kx}(\beta_{k}-\beta_{k-1})\Bigr\|+
\Bigl\|\bigvee_{k=m+1}^{r}{T^kx}(\beta_{k-1}-\beta_{k})\Bigr\|+\e K^2
\le
\e+4\e M_0 + \e K^2\to 0
$$
as $\e\to 0$.
Hence $1\in\sigma_{ap}(T)$.
\end{proof}

 
The proof above shows more. Namely, the following more general result is proved in the same way.
 
\begin{theorem} Let $X$ be a normed vector lattice, let $C\subset X_+$ be a non-zero max-cone. Let $T:C\to C$ be a mapping which is bounded, positively homogeneous and preserves finite suprema.
Let $C'\subset C$ be a bounded subset satisfying
$\|T^n\|=\sup\{\|T^nx\|: x\in C'\}$ for all $n$.
Then
$$
[\sup\{r_x(T):x\in C'\} , r(T) ]  \subset \sigma_{ap}(T).
$$
\label{appl_to_matrices}
\end{theorem}

Theorem \ref{main_lattice} and Lemma \ref{spdominate} imply the following result.
\begin{corollary} 
Let $X$ be 
a normed vector lattice and let $C\subset X_+$ be a non-zero max-cone. If $T:C\to C$ is a Lipschitz, positively homogeneous mapping which preserves finite suprema, then $r(T)= \max \{t : t \in \sigma_{ap}(T)\}$.
\label{rismax}
\end{corollary}
From Theorem \ref{main_lattice} also the following corollary follows. 
\begin{corollary} 
Let $X$ be a normed vector lattice and let $C\subset X_+$ be a non-zero max-cone. Let  $T:C\to C$ be a bounded, positively homogeneous mapping which preserves finite suprema. Then $r_x(T)\in\sigma_{ap}(T)$ for each $x\in C$, $x\ne 0$.
\label{local_lattice}
\end{corollary}

\begin{proof}
Let $x\in C$, $x\ne 0$. Let 
$$
K=\{\bigvee_{j=0}^n \alpha_j T^jx:n\in\NN, \alpha_j\ge 0 \quad(j=0,1,\dots, n)\}.
$$
Clearly $K$ is a non-zero max-cone, $TK\subset K$. Let $y\in K$, $y=\bigvee_{j=0}^n \alpha_j T^jx$ for some $n,\alpha_0,\dots,\alpha_n$.

Let $k\in\NN$. We have 
$$
\|T^ky\|\le
\sum_{j=0}^n \alpha_j\|T^{k+j}x\|\le
\max_j \alpha_j \cdot (n+1)\max\{\|T^{k+j}x\|:j=0,1,\dots,n\}
$$
and
$$
\|T^ky\|^{1/k}\le
(\max_j \alpha_j)^{1/k} \cdot (n+1)^{1/k}\max\{\|T^{k+j}x\|:j=0,1,\dots,n\}^{1/k}\to r_x(T)
$$
as $k\to\infty$. So 
$r_y(T)\le r_x(T).$ Thus $\sup\{r_y(T):y\in K\}=r_x(T)\le r(T|_{K})$.
By Theorem \ref{main_lattice}, $r_x(T)\in \sigma_{ap}(T|_{K})\subset\sigma_{ap}(T)$.
\end{proof}

\begin{remark}{\rm (i) As shown in Example \ref{good_stuff}, in Corollary \ref{rismax} the assumption "$T$ is Lipschitz"  can not be replaced by a weaker assumption "$T$ is bounded". 

(ii) An inspection of the proofs, or an application of the above results, show that Theorems \ref{main_lattice} and \ref{appl_to_matrices} and Corollaries \ref{rismax} and \ref{local_lattice}  hold under slightly weaker assumptions on $X$.  It suffices that $X$ is a ordered normed space, $C \subset X_+$ a non-zero max-cone and $X_+ - X_+$ a normed vector lattice (equivalently $X_+$ is a max-cone and there exists a lattice norm on $X_+ - X_+$).  
}
\end{remark}
Let us consider the following example from \cite{MN02}.
\begin{example} {\rm Let $X=C[0,1]$, $C=X_+$ and let $T:X \to X$ be a bounded linear operator defined by $T(x)(s)=sx(s)$.  The map $T:C \to C$ also preserves finite suprema (maxima) (and is Lipschitz and positively homogeneous) on $C$. 
As pointed out in \cite{MN02}, $r(T)=1$ and $T(x) \neq x$ for all $x\in C$, $x\neq 0$. 
However, $1 \in \sigma_{ap}(T)$ and the approximate sequence of vectors 
$(x_k)_{k\in \NN} \subset C$, $\|x_k\|=1$, is given by $x_k(s)=s^k$, since
$$\|Tx_k - x_k \| =  \frac{k^k}{(k +1)^{k+1}} \to 0$$ 
as $k\to \infty$.
\label{explicit}
}
\end{example}

\begin{remark}{\rm Under additional compactness type assumptions on $T$, Theorem \ref{main_lattice} implies Krein-Rutman type results. As is well known, and also shown by Example \ref{explicit}, some additional assumptions are necessary to obtain such results.
 Let $X$, $C$ and $T$ be as in Theorem \ref{main_lattice}, where $C$ is also closed. If, in addition, $T$ is  compact (and continuous) and $r(T)>0$, then  
there exists $y\in C$, $y\neq 0$ such that $Ty=r(T)y$. 
Moreover, for each nonzero $t\in\sigma_{ap}(T)$ there
exists an eigenvector in $C$.

Indeed, there exists a sequence $(x_k)\subset C$, $\|x_k\|=1$, with $Tx_k-tx_k\to
0$. Passing to a subsequence if necessary one can assume that $Tx_k\to
y$ for some $y\in C$. Clearly $tx_k\to y$, $y\ne 0$ and $Ty=ty$.

It is not hard to see that the same holds if we replace the assumption that "$T$ is compact" by the assumption that "$T$ is power compact" (i.e., that there exists $m \in \NN$ such that $T^m$ is compact).  

The results on the existence an eigenvector $x\in C$
 for a non-linear operator $T$ corresponding to $r(T)$ are known also under more general compactness type assumptions on $T$ (see e.g. \cite[Theorem 3.4, Theorem 3.10]{MN02}), \cite[Theorem 4.4]{MN10},
\cite[Theorem 10.6]{APV04}). We illustrate the usefulness of our 
Theorem \ref{appl_to_matrices} by giving an alternative  proof of \cite[Theorem 3.4]{MN02}
in the case of max-cones in Banach lattices and providing additional information in this case (Theorem \ref{MPrho<r}). Moreover, we do not need to assume the completeness of the norm.  In our proof we apply some of the ideas from \cite{F97}. On the other hand, the proof of 
\cite[Theorem 3.4]{MN02} was based on a lemma from fixed point index theory (\cite[Lemma 3.2]{MN02}, \cite[Theorem 2.1]{N81}).
     
To do this, we firstly recall some notions from \cite{MN02}. If $X$ is a normed space, let   $\nu$ denote
a homogeneous generalized measure of 
non-compactness on $X$ (as defined in \cite[Section 3]{MN02}), i.e., $\nu $ is a map which assigns to each bounded subset of $X$ a non-negative, finite number $\nu (A)$ and satisfies the following five conditions: 

(i) $\nu (A)=0$ if and only if $\overline{A}$ is compact,

(ii) $\nu (A+B) \le \nu (A) + \nu (B)$,

(iii) $\nu (\overline{ co (A)}) =\nu (A)$,

(iv) $\nu (A \cup B) = \max\{\nu (A), \nu (B)\}$,

(v) $\nu (\lambda A)= \lambda \nu (A)$ if $\lambda \ge 0$.

\noindent Here we denote $A+B =\{a+b: a\in A, b\in B\}$, and $\overline{ co (A)}$ denotes the smallest closed convex set containing $A$.

 Let $X$ be 
normed vector lattice, 
let $C\subset X_+$ be a non-zero closed max-cone and assume that $T:C\to X$ is a  continuous and
positively homogeneous mapping (and thus bounded). 
 Let
$$\nu _C (T)= \inf\{\lambda >0: \nu (T(A)) \le \lambda \nu (A) \;\; \mathrm{for} \;\; \mathrm{every} \;\; \mathrm{bounded} \;\; \mathrm{set} \;\; A \subset C\} \;\;$$ and
$$w_C (T)=  \sup\{\lambda >0: \nu (T(A)) \ge \lambda \nu (A) \;\; \mathrm{for} \;\; \mathrm{every} \;\; \mathrm{bounded} \;\; \mathrm{set} \;\; A \subset C\}, $$
where $\inf \emptyset =\infty$ and $\sup \emptyset =0$. Note that $w_C (I) = \nu _C (I) =1$ if $\mathrm{dim} (X)=\infty$. In this case we also have
\be
w_C (tI -T) \ge t -\nu _C (T) 
\label{w_nu}
\ee 
for $t\ge 0$. Indeed, 
$$t= w_C(tI) =  w_C(tI -T +T) \le  w_C(tI -T) + \nu _C ( T) , $$
which establishes (\ref{w_nu}).

 If $T:C \to C $, let
\be
\beta _{\nu} (T)= \lim _{n\to \infty} \nu _C (T^n)^{1/n}=\inf _{n\in \NN} \nu _C (T^n)^{1/n},
\label{pr_ess}
\ee
if $\nu _C (T^n) < \infty$ except for finitely many $n$. (If $\nu _C (T^n) = \infty$ for infinitely many $n$ one may define 
$\beta _{\nu}  (T)=\infty$.) The quantity $\beta _{\nu}  (T)$ was called the cone essential radius of $T$ in \cite{MN02}, but this terminology was changed in \cite{MN10} due to its imperfections.
}
\label{KrRut1}
\end{remark}
\begin{example}{\rm \cite[Examples on p. 14 and 15]{MN02} Let $X$ be a normed space and $A$ a bounded subset of $X$. By $\alpha (A)$ we denote the classical Kuratowski-Darbo generalized measure of noncompactness, i.e., 
$$\alpha (A)= \inf \{\delta >0: \mathrm{there} \;\; \mathrm{exist}\;\;  k\in \NN \;\; \mathrm{and} \;\;  S_i \subset X, i=1, \ldots, k $$
 $$ \mathrm{with}\;\;  \mathrm{diam}(S_i)\le \delta \;\; \mathrm{such} \;\; \mathrm{that}\;\; A= \cup_{i=1} ^k S_i  \},$$
where $\mathrm{diam}$ denotes the diameter of the set. 
In the case $X=C(W)$, where $(W, d)$ is a metric space, let us denote for $\delta >0$ 
$$\gamma _{\delta} (A) = \sup \{|x(t)-x(s)|: x \in A, \mathrm{with} \;\; t,s \in W \;\;\mathrm{satisfying} \;\; d(t,s) \le \delta\}.$$
Then 
$$\gamma (A) = \inf _{\delta >0} \gamma _{\delta} (A)  = \lim _{\delta \to 0^+}\gamma _{\delta} (A) $$ 
defines a generalized measure of noncompactness that satisfies $\alpha (A) \le \gamma (A) \le 2 \alpha (A)$. Consequently, $\beta _{\alpha} = \beta _{\gamma}$. However,  there exist nonequivalent measures of non-compactness (see  \cite{MN11a}, \cite{MN11b}) and this is one of the flaws of the quantity $\beta _{\nu}  (T)$.
}
\end{example}
The following result is a version of \cite[Theorem 3.4]{MN02} for max-cones in normed vector lattices, which  provides more information than  \cite[Theorem 3.4]{MN02} even for e.g. max-cones in Banach lattices.
\begin{theorem}
Let $X$ be 
a normed vector lattice with  $\mathrm{dim} (X)=\infty$ and  
let $C\subset X_+$ be a non-zero closed max-cone. 
Let $T:C\to C$ be a mapping which is continuous,
positively homogeneous and preserves finite suprema and  let $C'\subset C$ be a bounded subset satisfying
$\|T^n\|=\sup\{\|T^nx\|: x\in C'\}$ for all $n$.
Further, assume that $\nu$ is a homogeneous generalized measure of 
non-compactness on $X$. 

If $t \in [\sup\{r_x(T):x\in C'\} , r(T) ]$ satisfies $t > \beta _{\nu}  (T)$, 
then there exists 
a nonzero $x\in C$ such that $Tx = tx$.
\label{MPrho<r}
\end{theorem}
\begin{proof} By Theorem \ref{appl_to_matrices} we have $t \in \sigma _{ap} (T)$, so there exists a sequence $(x_k )_{k \in \NN} \subset C$, $\|x_k\|=1$, such that $\|(tI -T)x_k\| \to 0$ as $k\to \infty$. Denote $A=\{x_k : k \in \NN\}$.

First assume that $\nu _C (T) < t$. By (\ref{w_nu}) we have $w_C (tI -T) \ge t -\nu _C (T) >0 $. It follows from
$$w_C (tI-T) \nu (A) \le \nu ((tI-T)A)=0$$
that $\nu (A)=0$ and so $A$ has a compact closure. Therefore, there exist $z\in C$, $\|z\|=1$ and a subsequence $(x_{k_j}) \subset C$ such that $x_{k_j} \to z$ as $j \to \infty$ and so $Tz=tz$.

Since $\beta _{\nu} (T)  <t$, there exists $m\in \NN$ such that $\nu _C(T^m) < t^m$ by (\ref{pr_ess}). We also have
$$\sup\{r_x(T^m):x\in C'\} =(\sup\{r_x(T):x\in C'\})^m \le t^m \le  r(T)^m = r(T^m)$$
(see e.g. the proof of \cite[Proposition 2.1]{MN02}). By the above proved assertion there exists $y\in C$, $\|y\|=1$ such that $T^my =t^my$. Define $S=t^{-1}T $. Now the nonzero vector 
$$x= y\vee Sy \vee \cdots \vee S^{m-1}y \in C$$
satisfies $Sx=x$ and so $Tx=tx$.
\end{proof}

We call a max cone $C \subset X_+$ $\sigma$-order complete if for any $(x_n)_{n\in \NN} \subset C$, such that $x_n \le y$ for some $y\in X_+$  and all $n\in \NN$, there exists $\vee _{n=1} ^{\infty} x_n \in C$. 
In the following result we give some sufficient conditions for the existence of $x$ in a max-cone $C$ such that $r_x(T)=r(T)$.
 

\begin{proposition}
Let $X$ be a normed ordered space such that $X_+$ is complete. 
Let $C \subset X_+$  be a $\sigma$-order complete normal  max-cone and let $T:C\to C$ be a bounded, positively homogeneous, monotone mapping. 
Then there exists $x\in C$ such that $r_x(T)=r(T)$.
\label{attained}
\end{proposition}

\begin{proof}
The statement is trivial if $r(T)=0$. So without loss of generality we may assume that $r(T)=1$.
Then for each $k\in\NN$ we have $\|T^k\|\ge r(T^k)=1$. Find $x_k\in C$ such that $\|x_k\|=1$ and $\|T^kx_k\|\ge 1/2$. Set  $x=\bigvee_{k=1}^\infty k^{-2}x_k$. Then $x\in C$, 
$T^kx \ge k^{-2} T^k x_k$
and so
$$
M\|T^kx\|\ge k^{-2}\|T^kx_k\|\ge \frac{1}{2k^2},
$$
 where $M$ is the constant from the definition of a normal cone. 
So  \\ $r_x(T)=\limsup_k\|T^kx\|^{1/k}\ge 1=r(T)$.
Hence $r_x(T)=r(T)$.
\end{proof}

Our results can be applied to various max-type operators (and to the corresponding max-plus type operators and their tropical versions known also as Bellman operators) arising in diverse areas of mathematics and related applications
(see e.g. \cite{MN02}, \cite{KM97},  \cite{LM05}, \cite{LMS01},  \cite{AGN} and the references cited there). We point out the following example that was studied in detail in  \cite{MN02} and \cite{MN10}.
\begin{example} {\rm Given $a>0$, consider the following
max-type kernel operators $T:C[0,a] \to C[0,a]$ of the form
$$(T(x))(s)=\max_ {t\in [\alpha (s), \beta (s)]}{k(s,t)x(t)},$$
where $x\in C[0,a]$ and $\alpha, \beta:[0,a]\to[0,a]$ are given continuous functions satisfying $\alpha \le \beta$.
The kernel $k:\mathcal{S} \to [0, \infty)$ is a given non-negative continuous function, where $\mathcal{S}$ denotes the compact set
$$\mathcal{S}=\{(s,t)\in [0,a]\times[0,a]: t \in [\alpha (s), \beta (s)]\}.$$
It is clear that for $C=C_+ [0,a]$ it holds $TC \subset C$. We will denote the restriction $T|_C$ again by $T$.
The eigenproblem of these operators arises in the study of periodic solutions of a class of 
differential-delay equations
$$\varepsilon y^{\prime}(t)=g(y(t),y(t-\tau)), \quad \tau=\tau(y(t)),$$
with state-dependent delay (see e.g. \cite{MN02}). 

By \cite[Proposition 4.8]{MN02} and its proof the operator $T:C \to C$ is a positively homogeneous, Lipschitz map that preserves finite suprema. Hence by Corollary \ref{rismax} it follows  that $r(T)= \max \{t : t \in \sigma_{ap}(T)\}$. By \cite[Theorem 4.3]{MN02} it also holds that 
$r(T)= \lim _{n \to \infty} b_n ^{1/n} =\inf _{n\ge 1} b_n ^{1/n}$, where
$b_n =\|T^n\|= \max _{\sigma \in \mathcal{S}_n} k_n (\sigma)$,
$$k_n  (\sigma)=k(s_0, s_1)k(s_1, s_2) \cdots k (s_{n-1}, s_n) $$
and
$$\mathcal{S}_n =\{(s_0, s_1,s_2,  \ldots , s_n): s_0 \in [0,a], s_i \in [\alpha(s_{i-1}), \beta (s_{i-1})], i=1,2, \ldots , n \}$$ 

Recall that certain Krein-Rutman type results were proved for $T:C \to C$ in \cite[Theorems 4.1, 4.2 , 4.4, Corollaries 4.21, 4.22]{MN02}, i.e.,  under suitable additional conditions on $\alpha, \beta$ and $k$ (suitable generalized compactness type conditions on $T$), there exists $x \in C$, $x\neq 0$, such that $Tx=r(T)x$. However, it was also shown in \cite[Proposition 4.23]{MN02} that there are also reasonable conditions on  $\alpha, \beta$ and $k$ for which such an eigenvector $x$ does not exist.
}
\end{example}

We also consider the following related example.
\begin{example}{\rm 
Let $M$ be a nonempty set and let $X$ be the set of all bounded real functions on $M$.
With the norm $\|f\|_{\infty}=\sup\{|f(t)|:t\in M\}$ and natural operations $X$ is a normed vector lattice.
Let $C=X_+$ be the positive cone.

Let $k:M\times M\to [0,\infty)$ satisfy $\sup \{k(t,s):t,s\in M\}<\infty$.

Let $T:C\to C$ be defined by $(Tf)(s)=\sup\{k(s,t)f(t):t\in M\}$ and so $\|T\|=\sup \{k(t,s):t,s\in M\}$. Clearly $C$ is a max-cone, $T$ is bounded, positive homogeneous and preserves finite maxima. So Theorem \ref{main_lattice} applies.
Moreover, $T$ is Lipschitz. So by Corollaries \ref{rismax} and \ref{local_lattice} we have that $r(T)= \max \{t : t \in \sigma_{ap}(T)\}$ and $r_x(T)\in\sigma_{ap}(T)$ for each $x\in C$, $x\ne 0$. Note also that $r(T)=r_e (T)$, where $e(t)=1$ for all $t \in M$.

In particular, if $M$ is the set of all natural numbers $\NN$, our results apply to infinite bounded non-negative matrices $A=[a(i,j)]$ (i.e., $a(i,j) \ge 0$ for all $i,j \in \NN$ and    $\|A\|=\sup _{i,j \in \NN} a(i,j) < \infty$). In this case, $X= l^{\infty}$ and 
$C=l^{\infty}_+ $. We denote $T_A =T$ and we have
$$\|T_A\|= \|A\|= \sup _{j\in \NN} \|T_A e_j\|,$$
where $\{e_j: j\in\NN\}$ is the set of standard basis vectors. By Theorem \ref{appl_to_matrices} the following result follows.
}
\label{bounded}
\end{example}

\begin{corollary} Let $A$ be an infinite bounded non-negative matrix and 
let \\
$\sup\{r_{e_j}(T_A):j\in\NN\}\le t\le r(T_A)$.
Then $t\in \si_{ap}(T_A)$.
\end{corollary}
\begin{proof} The set $C'=\{e_j: j\in\NN\}$ satisfies the conditions of Theorem \ref{appl_to_matrices}, which gives the result.
\end{proof}

The following example shows that in general $\sup\{r_{e_j}(T_A):j\in\NN\} \ne r (T_A)$. 
\begin{example}
{\rm
 Consider the left (backward) shift $T_A: C \to C$, $C=l^{\infty}_+$, 
 $$T_A(x_1, x_2, x_3, \ldots) = (x_2, x_3, x_4, \ldots),$$
 i.e., $T_A e_1 = 0$ and $ T_A e_j = e_{j-1}$ for all $j \ge 2$. Then $r_{e_j}(T_A)=0$ for each basis
element $e_j$, but $r(T_A)=1$. We also have $\si_{ap}(T_A)=[0,1]= \si_{p} (T_A)$, where
$$\si_{p} (T_A) =\{ t\ge 0: T_A x= tx\;\; \mathrm{for}\;\; \mathrm{some}\;\; x\in C, \|x\|=1\}. $$

On the other hand, for its restriction $T_A |_{c_0 ^+ }$, to the positive cone of the space of null convergent sequences $c_0 ^+$,  we have $r(T_ A|_{c_0 ^+})=1$, $\si_{p} (T_A |_{c_0 ^+}) =[0,1)$ and $ \si_{ap}(T_ A|_{c_0 ^+})=[0,1]$. Note that this again shows, in particular, that some compactness type assumptions in 
Remark \ref{KrRut1} and in Theorem \ref{MPrho<r} are necessary.
}
\label{leftshift}
\end{example}


\begin{remark}{\rm The special case of Example \ref{bounded} when  
$M=\{1, \ldots , n\}$ for some $n\in \NN$ is well known and studied under the name max-algebra (an analogue of linear algebra). Together with  its isomorphic versions (max-plus algebra and min-plus algebra also known as tropical algebra) it provides 
an attractive way of describing a class of non-linear problems appearing for instance in manufacturing and transportation scheduling, information 
technology, discrete event-dynamic systems, combinatorial optimization, mathematical physics, DNA analysis, ... (see e.g.  \cite{B98}, \cite{Bu10}, \cite{MP15}, \cite{MP12} 
\cite{BCOQ92},  \cite{BGC-G09}, \cite{PS05} and the references cited there). 

In particular, for a non-negative $n\times n$ matrix $A$  it holds (see e.g. \cite[Theorem 2.7]{MP15}) that 
$$
\si_ {ap}(T_A)=\si_ {p}(T_A) =\{t: \hbox{ there exists } j\in\{1,\dots,n\}, t=r_{e_j}(T_A)\}.
$$
However, as Example \ref{leftshift} shows, an analogue of  this result is not valid for infinite bounded non-negative matrices.
}
\end{remark}

\section{ Results on normal convex cones in normed spaces}

Next we prove the analogues of Theorems \ref{main_lattice} and \ref{appl_to_matrices} for positively homogeneous, additive and Lipschitz maps defined on normal wedges in normed spaces. This result generalizes and extends an implicitly known result that for a positive linear operator $T$ on a Banach lattice $X$ there exists a  positive sequence of approximative vectors for the usual spectral radius (see e.g. the proof of Krein-Rutman's theorem \cite[Theorem 7.10]{AA02}). We do not  assume the completeness of the norm, we do not assume that the space $X$ is a lattice and we also do not need to assume that the wedge $C$ is closed.
In the proof we apply a technique of the proof of Theorem \ref{main_lattice} and we include it for the sake of completeness.
\begin{theorem} 
Let $X$  be a normed space, $C\subset X$ a non-zero normal wedge and let $T:C\to C$ be positively homogeneous, additive and Lipschitz.

 If
$$
\sup\{r_x(T):x\in C\}\le t\le r(T),
$$
then $t\in\sigma_{ap}(T)$.

In particular,
$r(T)\in\sigma_{ap}(T)$. Moreover, $r(T)= \max \{t : t \in \sigma_{ap}(T)\}$.
\label{main_normal}
\end{theorem}

\begin{proof} By Lemma \ref{spdominate} we have $r(T) \ge t$ for all  $t \in \sigma_{ap}(T)$.

Let $\sup\{r_x(T):x\in C\}\le t\le r(T)$.
Similarly as in the proof of Theorem \ref{main_lattice}  we may assume  without loss of generality that $t=1$.

\medskip
We distinguish two cases:

\noindent{\bf I.}
Suppose that $\sup\{\|\sum_{j=0}^n T^jx\|:x\in C, \|x\|=1,n\in\NN\} =\infty$. 

Let $k\in\NN$. Find $n_k$ and $x_k\in C$ such that $\|x_k\|=1$ and $\|\sum_{j=0}^{n_k} T^jx_k\|>k$.
Find $t_k\in(1,1+k^{-1})$ such that
$t_k^{-n_k-1}> 1/2$. Find $r_k>n_k$ such that $\frac{\|T^{r_k}x_k\|}{t_k^{r_k}}<1$. Set $y_k:=\sum_{j=0}^{r_k}\frac{T^jx_k}{t_k^{j+1}}$. Then $M\|y_k\|\ge
\|t_k^{-n_k-1}\sum_{j=0}^{n_k} T^jx_k \|\ge k/2$,  where $M$ is the constant from the definition of a normal cone.

Set $u_k=\frac{y_k}{\|y_k\|}$. Then $\|u_k\|=1$. Since $T$ is additive, positively homogeneous and Lipschitz with a  Lipschitz constant $L$, we have
$$
\|Tu_k-u_k\|\le
\|Tu_k-t_ku_k\|+(t_k-1)\|u_k\|=
\|y_k\|^{-1}\Bigl\| T\left(\sum_{j=0}^{r_k}\frac{T^{j}x_k}{t_k^{j+1}}\right) -\sum_{j=0}^{r_k}\frac{T^{j}x_k}{t_k^{j}}\Bigr\|+ (t_k-1) $$
$$\le  \|y_k\|^{-1} \left (\Bigl\| T\left(\sum_{j=0}^{r_k}\frac{T^{j}x_k}{t_k^{j+1}}\right) -T \left(\sum_{j=1}^{r_k}\frac{T^{j-1}x_k}{t_k^{j}}\right) \Bigr\| +1 \right)+ (t_k-1) $$
$$\le \frac{2M}{k} \left ( L\Bigl\| \sum_{j=0}^{r_k}\frac{T^{j}x_k}{t_k^{j+1}} -\sum_{j=1}^{r_k}\frac{T^{j-1}x_k}{t_k^{j}}\Bigr\| +1 \right)+ k^{-1} =\frac{2M}{k} \left (L\Bigl\| \frac{T^{r_k}x_k}{t_k^{r_k+1}} \Bigr\| +1 \right)+ k^{-1}$$
$$<
\frac{2M (L+1) +1}{k}\to 0
$$
as $k\to\infty$.
Hence $1\in\sigma_{ap}(T)$.

\medskip

\noindent{\bf II.}
Suppose that $M_0:=\sup\{\|\sum_{j=0}^nT^jx\|:x\in C,\|x\|=1,n\in\NN\}<\infty$.

Let $\e>0$ and $K=\|T\|$. Let $n\in\NN$ be the number constructed in Lemma \ref{numbers}.
We have $\|T^n\|\ge r(T^n)=r(T)^n\ge 1$, so there exists $x\in C$ such that $\|x\|=1$ and $\|T^nx\|\ge 1/2$. Write $\al_k=\|T^kx\|$.

Let $m\in\NN$ and $\beta_k\quad(k\ge 0)$ be the numbers constructed in Lemma \ref{numbers} , i.e., $\beta_k\ge 0$, $\beta_0\le\e$, $|\beta_{k+1}-\beta_k|\le 2\e$, $ \beta _k < \beta _{k+1} $  for  all  $k=0, 1, \ldots , m-1$,
$ \beta _k > \beta _{k+1} $ for  all $ k=m, m+1, \ldots $,
$\al_k\beta_k\le 1$, $\al_m\beta_m=1$, $\lim_{k\to\infty}\al_k\beta_{k+1}=0$.

Fix $r>m$
such that $\beta_{r}<\e\|T^{r-1}x\|^{-1}$.

Set $u=\sum_{k=0}^r\beta_k T^kx$. Since $u\ge \beta_m{T^mx}$, we have $M\|u\|\ge \beta_m{\|T^mx\|}=1$.
We have
$$
\|Tu-u\|
=
\Bigl\| T\left(\sum_{k=0}^r\beta_k{T^{k}x}\right)-
\sum_{k=0}^r\beta_k{T^kx}\Bigr\| 
$$
$$\le \Bigl\| T\left(\sum_{k=0}^r\beta_k{T^{k}x}\right)-
\sum_{k=1}^r\beta_k{T^kx}\Bigr\| + \Bigl\|\beta_0{x}\Bigr\|
$$
$$\le  L \Bigl\| \sum_{k=0}^r\beta_k{T^{k}x}-\sum_{k=1}^r\beta_k{T^{k-1}x}\Bigr\| + \e$$
$$
\le
 L \left (\Bigl\|\beta_r{T^{r}x}\Bigr\|+ \Bigl\|\sum_{k=0}^{m-1}{T^kx}(\beta_{k+1}-\beta_{k})\Bigr\|+
\Bigl\|\sum_{k=m}^{r-1}{T^kx}(\beta_{k}-\beta_{k+1})\Bigr\| \right)+ \e
$$
$$\le
 L(K\e+4\e M_0 M)+\e\to 0
$$
as $\e\to 0$.
Hence $1\in\sigma_{ap}(T)$.
\end{proof}
Similarly as Theorem \ref{appl_to_matrices}, the following more general result is proved in a similar way. 
\begin{theorem} Let $X$  be a normed space, $C\subset X$ a non-zero normal wedge and let $T:C\to C$ be positively homogeneous, additive and Lipschitz.
Let $C'\subset C$ be a bounded subset satisfying
$\|T^n\|=\sup\{\|T^nx\|: x\in C'\}$ for all $n$.
Then
$$
[\sup\{r_x(T):x\in C'\} , r(T) ]  \subset \sigma_{ap}(T)
$$
and $r(T)= \max \{t : t \in \sigma_{ap}(T)\}$.
\label{appl_to_matrices_lin}
\end{theorem}

\begin{remark}{\rm In fact, in the setting of Theorems \ref{main_normal} and \ref{appl_to_matrices_lin} the map $T$ extends to a bounded linear operator on the normed space $C-C$. Note that in general $C-C$  is not a lattice. 
Moreover, the obtained approximate eigenvectors are in $C$.
}
\end{remark}

\begin{corollary}  Let $X$ be a normed space with a non-zero normal wedge $C\subset X$. 
If $T:C\to C$ is positively homogeneous, additive and Lipschitz, then
$r_x(T)\in\sigma_{ap}(T)$ for each $x\in C$.
\end{corollary}
\begin{proof}
Let $x\in C$ and let $K$ be the smallest wedge generated by $x$ invariant for $T$, i.e., $K=\{\sum_{j=0}^n\al_jT^jx:n\in\NN,\al_0,\dots,\al_n\ge 0\}$.

Let $y=\sum_{j=0}^n\al_jT^jx\in K$. Then $\|T^ky\|\le\sum_{j=0}^n\al_j\|T^{k+j}x\|$ and it is easy to see that $r_y(T|_{K})\le r_x(T) \le r(T|_{K}) $. By Theorem \ref{main_normal}, $r_x(T)\in\sigma_{ap}(T|_{K})\subset\sigma_{ap}(T)$.
\end{proof}

\begin{remark}{\rm Let $X$, $C$ and $T$ be as in Theorem  \ref{appl_to_matrices_lin}. Under additional compactness type assumptions from Remark \ref{KrRut1} or Theorem \ref{MPrho<r}, Theorem  \ref{appl_to_matrices_lin} implies Krein-Rutman type results, i.e., 
 the existence of an eigenvector $x\in C$, $x\neq 0$ such that $Tx=tx$. We omit the details.
 As is well-known and also illustrated by Examples \ref{explicit} and \ref{leftshift}, such additional assumptions are necessary.
}
\end{remark}

\bigskip

The following result is essentially known, see e.g.  \cite[Theorem 3.3]{MN10}, \cite[Theorem 2.1]{Gr15} and  \cite[Theorem 2.2 and remarks after it]{MN02}. The sketch of the proof is included for the sake of completeness.

\begin{proposition} Let $X$ be a normed 
space and $C\subset X$ a non-zero normal complete wedge. 
Let $T:C\to C$ be  bounded and positively homogeneous. 
If, in addition, 

(i) $T$ is monotone on $C$ or

(ii) $T$ is continuous and additive on $C$,

\noindent then
there exists $x\in C$ such that $r_x(T)=r(T)$.
\label{att2}
\end{proposition}
\begin{proof}
Without loss of generality we assume that $r(T)=1$ and for each $k\in\NN$ choose $x_k\in C$ as in the proof of Proposition \ref{attained}. Define $x=\sum_{k=1}^\infty k^{-2}x_k \in C$. 
If $T$ satisfies (i) or (ii), it follows that $Tx \ge  k^{-2}Tx_k$. Conclude the proof as in the proof of Proposition \ref{attained}. 
\end{proof}
\begin{remark}{\rm Similarly as in \cite[Remark on p.12]{MN02}, a slight generalization of Proposition  \ref{att2} is possible.
Namely, if $C_1 \subset C$ are given wedges and $T$ satisfies the conditions of  Theorem \ref{att2} with respect to the wedge $C$ as stated. Additionally, we assume that $TC_1 \subset C_1$, where the wedge  $C_1$ is complete.  Then there exists $x \in C_1$ that equals the Bonsall cone spectral radius of $T$ with respect to $C_1$.

As pointed out (and applied) in \cite{MN02} and \cite[Theorem 3.3]{MN10}, the main reason for this generalization is that it may happen that a non-linear map is   monotone with respect  to the (pre)ordering $\le _{C}$, but it is not monotone with respect  to the (pre)ordering $\le _{C_1}$ (see, for instance, \cite{LN08} and the "renormalization operators" which occur in discussing diffusion on fractals).
}
\end{remark}

\vspace{3mm}

\baselineskip 5mm

{\it Acknowledgments.}
The first author was supported by grants No. 14-07880S of GA
CR and RVO: 67985840.

The second author was supported in part by the JESH grant of the Austrian Academy of Sciences and by grant P1-0222 of the Slovenian Research Agency. The second author
 thanks  Marko Kandi\'{c} and Roman Drnov\v{s}ek for useful comments and to his collegues and staff at TU Graz for their hospitality during his stay in Austria. \\

\vspace{2mm}

\noindent
Vladimir M\"uller\\
Institute of Mathematics, Czech Academy of Sciences \\
\v{Z}itna 25 \\
115 67 Prague, Czech Republic\\
email: muller@math.cas.cz

\bigskip

\noindent
Aljo\v sa Peperko \\
Faculty of Mechanical Engineering \\
University of Ljubljana \\
A\v{s}ker\v{c}eva 6\\
SI-1000 Ljubljana, Slovenia\\
{\it and} \\
Institute of Mathematics, Physics and Mechanics \\
Jadranska 19 \\
SI-1000 Ljubljana, Slovenia \\
e-mails : aljosa.peperko@fmf.uni-lj.si , aljosa.peperko@fs.uni-lj.si

\end{document}